\documentclass[12pt]{article}

\usepackage{amsmath,amsthm,amssymb,amsfonts,amscd,mathrsfs, mathtools, dsfont}
\usepackage[hyperindex,breaklinks]{hyperref}
\usepackage{cleveref}
\usepackage{tikz}
\usepackage{tikz-cd}
\usepackage{cancel}
\usepackage{fullpage}
\usepackage{enumerate}
\usepackage[shortlabels]{enumitem}
\usepackage{makecell}
\usetikzlibrary{matrix,arrows,decorations.pathmorphing}
\setlist[itemize]{noitemsep, topsep=0pt}

\theoremstyle{definition}
\newtheorem{theorem}{Theorem}[section]
\newtheorem{corollary}[theorem]{Corollary}
\newtheorem{lemma}[theorem]{Lemma}

\newtheorem{definition}[theorem]{Definition}

\newtheorem{remark}[theorem]{Remark}

\numberwithin{equation}{section}
\numberwithin{theorem}{section}

\definecolor{darkgreen}{rgb}{0,0.5,0}
\definecolor{darkblue}{rgb}{0,0.1,0.5}

\hypersetup{
    colorlinks=true, 
    linkcolor=darkblue, 
    urlcolor=blue, 
    linktoc=all, 
    citecolor=darkgreen
    }

\newcommand{\Id}{\mathrm{Id}}

\newcommand{\Hom}{\mathrm{Hom}}

\newcommand{\mcC}{\mathcal{C}}
\newcommand{\mcD}{\mathcal{D}}

\newcommand{\mcF}{\mathcal{F}}
\newcommand{\mcG}{\mathcal{G}}

\newcommand{\mcI}{\mathcal{I}}
\newcommand{\mcJ}{\mathcal{J}}

\DeclareMathOperator{\Rep}{Rep}

\newcommand{\Rad}{\mathrm{Rad}}
\newcommand{\Soc}{\mathrm{Soc}}
\newcommand{\rad}{\mathrm{rad}}
\newcommand{\soc}{\mathrm{soc}}

\title{Preservation of Loewy Diagrams Under Exact Functors}
\author{Matthew Rupert}
\date{}

\begin{document}

\maketitle

\abstract{We derive sufficient conditions for exact functors on locally finite abelian categories to preserve Loewy diagrams of objects. We apply our results to determine sufficient conditions for induction functors associated to simple current extensions of vertex algebras to preserve Loewy diagrams.}

\setlength\parindent{0pt}
\setlength{\parskip}{\baselineskip}%

\section{Introduction}

The Loewy diagram of an object is a diagram whose layers consist of the semisimple quotients of a Loewy series annotated with arrows indicating the existence of a non-split subquotient. Such diagrams are useful tools for determining the structure of abstractly defined objects such as projective covers. They have been used to investigate the structure of modules over many different algebras including, but not limited to, vertex algebras \cite{CCG}, Temperley-Lieb algebras \cite{BRS}, and quantum groups \cite{CRR}. For a Loewy diagram with Loewy length two, it is sometimes possible to show directly that an exact functor preserves the Loewy diagram as in \cite[Figure 1]{CKL}. These techniques fail, however, for Loewy diagrams of length greater than two, such as those which appear in \cite{CRR}. It is therefore desirable to have a general criterion for when an objects Loewy diagram is preserved under the action of an exact functor. 

The authors primary motivation for studying the preservation of Loewy diagrams under exact functors stems from the theory of vertex algebra extensions. Many interesting examples of vertex algebra extensions can be realized as simple current extensions, that is, they decompose as a direct sum of invertible modules (or simple currents). For example, the triplet $W_Q(p)$ is a simple current extension of the singlet $W_Q^0(p)$, where $p \in \mathbb{Z}_{\geq 2}$ and $Q$ is the root lattice of a simply laced simple finite dimensional complex Lie algebra, and the $B_Q(p)$ vertex operator algebras are simple current extensions of $W^0_Q(p) \otimes \mathsf{H}$ where $\mathsf{H}$ is the Heisenberg vertex algebra. One method for lifting information from a vertex algebra to its simple current extension is through induction functors. In particular, consider the following powerful result:
\begin{theorem}\cite{CKL,CKM,HKL}\\
Let $\mathcal{C}$ be a category of modules over some vertex operator algebra $V$ with a natural vertex tensor category structure. Then the following are equivalent:
\begin{itemize}
\item A vertex operator algebra extension $V \subset V^{e}$ with $V^e \in \mathcal{C}$ as a $V$-module.
\item A commutative associative algebra in $\mathcal{C}$ with trivial twist and injective unit.
\end{itemize}
Further, the category of $V^e$-modules which lie in $\mathcal{C}$ as $V$-modules is braided equivalent to $\mathrm{Rep}^0V^e$, now viewing $V^e \in \mathcal{C}$ as a commutative algebra object.
\end{theorem}
Commutative algebra objects admit induction functors (see Equation \eqref{Eq:Ind}) which preserve projectivity so knowing when such a functor preserves Loewy diagrams provides a convenient criteria to lift information about the structure of projective modules over a vertex algebra to its simple current extensions. 

A particularly interesting example of a simple current extension is the simple affine VOA $L_{-\frac{3}{2}}(\mathfrak{sl}_3)$, which is a simple current extension of $W^0_{A_2}(2) \otimes \mathsf{H}$ and one of the most accessible examples beyond rank $1$. Conjectures for the structure of Loewy diagrams of projective covers for $L_{-\frac{3}{2}}(\mathfrak{sl}_3)$ were given in \cite[Figure 7]{CRR}. Implicit in the construction of these conjectures is (among other things) the assumption that the associated induction functor preserves Loewy diagrams. We investigate in Section \ref{Sec:Loewy} the general question of when an exact functor preserves Loewy diagrams and obtain the following result (Theorem \ref{Thm:Soc} and Corollary \ref{Cor:Loewy}):
\begin{theorem}
Let $\mcC$ be a locally finite abelian category and $\mcF:\mcC \to \mcD$ a strong exact functor (Definition \ref{Def:strong}) which preserves socle (or radical) filtrations (Definition \ref{Def:pres}), then $\mcF:\mcC \to \mcD$ preserves the associated Loewy diagrams.
\end{theorem}
In Section \ref{Sec:Ind} we apply this result to induction functors associated to commutative algebra objects which are direct sums over a collection of invertible objects $\{U_i\}_{i \in \mcI}$:
\[ A:=\bigoplus\limits_{i \in \mcI} U_i.\]
We refer to such algebra objects as commutative simple current algebra objects when the collection $\{U_i\}_{i \in \mcI}$ is closed under products and duals (see Subsection \ref{subsec:inv} for details). We say $A$ has no fixed points if the collection $\{U_i\}_{i \in \mcI}$ has no fixed points (see Definition \ref{Def:fp}). We obtain the following result (Theorem \ref{Thm:Ind}):
\begin{theorem}
Let $\mcC$ be a braided locally finite abelian category and $A \in \mcC$ a commutative simple current algebra object with no fixed points which is simple as a left $A$-module. Then the induction functor $\mcF_A:\mcC \to \Rep A$ preserves Loewy diagrams.
\end{theorem}

\subsection*{Acknowledgements}
The author is grateful to Thomas Creutzig for his helpful comments. The author is supported by the Pacific Institute for the Mathematical Sciences (PIMS) postdoctoral fellowship
program.

\section{Preliminaries}
Throughout we assume that $\mcC$ is an abelian category. Given a monomorphism $f:X \to Y$ in $\mcC$, we adopt the standard notation $Y/X := \mathrm{Coker}(f)$. Cokernels always exist in abelian categories and are preserved by exact functors. If $\mcF:\mcC \to \mcD$ is exact, then $\mcF(\mathrm{Coker}(f)) \cong \mathrm{Coker}(\mcF(f))$ which can be represented in quotient notation as $\mcF(Y/X) \cong \mcF(Y)/\mcF(X)$.

\subsection{Loewy Diagrams}

Recall that a series of an object $M \in \mcC$ is a strictly increasing family of subobjects ordered by inclusion:
\[ 0=M_0 \subset M_1 \subset \hdots \subset M_n=M \qquad \mathrm{or} \qquad 0=M_n \subset M_{n-1} \subset \hdots M_0=M.\]
The successive quotients of such a series are the quotients $M_1/M_0, M_2/M_1,...,M/M_{n-1}$ or $M_0/M_1,...,M_{n-1}/M_n$  and the integer $n$ is called the length of the series.
\begin{definition}
A Loewy series for $M \in \mcC$ is a series of minimal length amongst those series which have semisimple successive quotients.
\end{definition}
Given an object $M \in \mcC$, the socle of $M$, $\mathrm{Soc}(M)$ is the largest semisimple subobject of $M$. That is, the direct sum of all simple subobjects of $M$. The radical of $M$, $\mathrm{Rad}(M)$, is the intersection of all maximal subobjects of $M$. There are two standard examples of Loewy series called the \textit{socle} and \textit{radical} series of $M \in \mcC$. The \textit{radical} series of $M \in \mcC$ 
\[ 0=\rad_n(M) \subset \rad_{n-1}(M) \subset \hdots \subset \rad_1(M) \subset \rad_0(M)=M\]
is defined inductively with $\rad_0(M)=M$ and $\rad_k(M)=\mathrm{Rad}(\rad_{k-1}(M))$. The \textit{socle} series of $M \in \mcC$
\[ 0=\soc_0(M) \subset \soc_1(M) \subset \hdots \subset \soc_{n-1}(M) \subset \soc_n(M)=M\]
 is defined inductively by $\soc_1(M)=\Soc(M)$ and letting $\soc_k(M)$ be the unique subobject of $M$ such that 
\begin{equation}\label{eq:Soc}\Soc(M/\soc_{k-1}(M)) \cong \soc_k(M)/\soc_{k-1}(M).\end{equation}

\begin{definition}\label{Def:Loewy}
 Given a fixed Loewy series $ 0=M_0 \subset M_1 \subset \hdots \subset M_n=M$ of $M \in \mcC$, the associated Loewy diagram is a diagram consisting of horizonatal layers where the $k$-th layer from the bottom consists of the simple factors in $M_k/M_{k-1}$. Loewy diagrams are annotated with vertical arrows from an irreducible factor $X_k$ in layer $k$ to $X_{k-1}$ in layer $k-1$ if there exists a subquotient X of $M_k/M_{k-2}$ such there exists a non-split short exact sequence
\[0 \to X_{k-1} \to X \to X_k \to 0.\]
\end{definition}

\begin{remark}\label{Rem}
It is implicit in the definition of the socle filtration that the subobjects $\mathrm{soc}_k(M)$ of $M$ are uniquely embedded in $M$ in the sense that if $f,g:\soc_k(M) \to M$ are any two monomorphisms, then $\mathrm{Im}(f)=\mathrm{Im}(g)$ and therefore $\mathrm{Coker}(f) \cong \mathrm{Coker}(g)$. This is obviously true for $\soc_1(M)=\Soc(M)$ and can be seen inductively for $k>1$ by observing that if $\mathrm{Im}(f) \not = \mathrm{Im}(g)$ then there is a canonical morphism $\phi:\mathrm{Im}(f) \oplus \mathrm{Im}(g) \to M$ whose image $\mathrm{Im}(\phi)$ is a subobject of $M$, contains $M_{k-1}$ as a subobject, and whose (unique) quotient by $M_{k-1}$ is semisimple, contradicting Equation \eqref{eq:Soc}.
\end{remark}

\subsection{Algebra Objects}
We now provide the necessary background on commutative algebra objects in braided categories. For a more detailed treatment, see \cite{CKM,KO}. In what follows, let $\mcC$ be a braided category with braiding $c_{-,-}$, associativity isomorphism $a_{-,-,-}$, and left unit isomorphism $\ell_{-}$.
\begin{definition}
An associative algebra in $\mcC$ is an object $A \in \mcC$ equipped with maps $\mu:A \otimes A \to A$ and $\iota:\mathds{1} \to A$ such that
\begin{itemize}
\item $\mu \circ (\mu \otimes \mathrm{Id}_A) \circ a_{A,A,A}=\mu \circ (\mathrm{Id}_A \otimes \mu)$. \hfill Associativity
\item $\mu \circ (\iota \otimes \mathrm{Id}_A)\circ \ell_A^{-1}=\mathrm{Id}_A$. \hfill Unit
\end{itemize}
Such an algebra is called commutative if $\mu \circ c_{A,A} = \mu$.
\end{definition}
We define the category $\mathrm{Rep}A$ as objects $(M,\mu_M)$ where $M \in \mcC$ and $\mu_M:A \otimes M \to M$ is a morphism in $\mcC$ such that
\begin{itemize}
\item $\mu_M \circ (\mu \otimes \mathrm{Id}_M) \circ a_{A,A,M}=\mu_M \circ (\mathrm{Id}_A \otimes \mu_M)$ \hfill Associativity
\item $\mu_M \circ (\iota \otimes \mathrm{Id}_A) \circ \ell_M^{-1}=\mathrm{Id}_M$ \hfill Unit
\end{itemize}
A morphism $f:(M,\mu_M) \to (N,\mu_N)$ in $\mathrm{Rep}A$ is a $\mcC$-morphism $f: M \to N$ such that $f \circ \mu_M=\mu_N \circ (\mathrm{Id}_A \otimes f)$, that is, $f$ intertwines the action of $A$ on $M$ and $N$. There exists a functor $\mcF_A:\mcC \to \mathrm{Rep}A$ given by 
\begin{equation}\label{Eq:Ind} \mcF_A(M)= (A \otimes M,(\mu \otimes \mathrm{Id}_M) \circ a_{A,A,M}) , \qquad \mcF_A(f)=\mathrm{Id}_A \otimes f\end{equation}
called the induction functor. If $\mcG:\mathrm{Rep}A \to \mcC$, $(M,\mu_M) \mapsto M$ is the forgetful functor, then there exists a natural isomorphism \cite[Lemma 7.8.12]{EGNO}
\begin{equation}\label{Eq:Frob}
\mathrm{Hom}_{\mathrm{Rep}A}(\mcF_A(M),N) \cong \mathrm{Hom}_{\mcC}(M,\mcG(N))
\end{equation}
which is often referred to as Frobenius reciprocity.

\subsection{Invertible Objects}\label{subsec:inv}
In Section \ref{Sec:Ind} we will derive sufficient conditions for induction functors of commutative algebra objects which are direct sums of invertible objects to preserve Loewy diagrams. We collect here the properties of invertible objects which we will require. Given a monoidal category $\mcC$ with unit $\mathds{1}$, an object $U \in \mcC$ is called invertible if it is rigid (sometimes referred to as dualizable) such that the (co)-evaluation morphisms
\[ \mathrm{ev}_{U}:U^* \otimes U \to \mathds{1}, \qquad \mathrm{coev}_U:\mathds{1} \to U \otimes U^*,\]
are isomorphisms. The subcategory of invertible objects is closed under duals and tensor products \cite[Proposition 2.11.3]{EGNO}. Given a family of invertible objects $\{U_i\}_{i \in \mcI}$, we say the family is closed if it is closed under tensor products and inverses (duals). That is, if for all $i,j \in \mcI$, $U_i \otimes U_j \cong U_k$ for some $k \in \mcI$ and $U_{-i}\cong U_s$ for some $s \in \mcI$ where we now denote the inverse of $U_i$ by $U_{-i}:=U_i^*$.
\begin{definition}\label{Def:fp}
A fixed point of an invertible object $U \in \mcC$ is an object $N \in \mcC$ such that $U \otimes N \cong N$. We say a family of invertible objects $\{U_i\}_{i \in \mcI}$ has no fixed points if no member of the family other than the monoidal unit has a fixed point.
\end{definition}
We collect in the following lemma some well known properties of invertible objects which we will require.
\begin{lemma}\label{Lem:inv}
Let $\mcC$ be a monoidal category. Then we have the following:
\begin{itemize}
\item[\textbf{(1)}] If $\{U_i\}_{i \in \mathcal{I}}$ is a closed family of invertible objects in $\mcC$ with no fixed points, then $U_i \otimes N \not \cong U_j \otimes N$ for all $i,j \in \mathcal{I}$ and $N \in \mathcal{C}$.
\item[\textbf{(2)}] If $\mcC$ is abelian and $M,U \in \mcC$ with $M$ simple and $U$ invertible, then $U \otimes M$ and $M \otimes U$ are simple.
\end{itemize}
\end{lemma}
The proof of $(1)$ is a straight-forward proof by contradiction. To see that $(2)$ holds, first observe that if the functors $(U \otimes -)$ and $(- \otimes U)$ for any invertible object $U \in \mcC$ are exact then the statement follows easily. Indeed, if $U \otimes M$ is not simple it admits some proper subobject $N$ satisfying an exact sequence
\[ 0 \to N \to U \otimes M. \] 
Then exactness of $U^* \otimes -$ yields an exact sequence
\[ 0 \to U^* \otimes N \to U^* \otimes U \otimes M \cong M.\]
Simplicity of $M$ now gives $U^* \otimes N \cong M$ so $N \cong U \otimes M$ by invertibility of $U$. Now we need only check that the functors given by tensoring with invertible objects are always exact. First, recall from \cite[Proposition 2.11.3]{EGNO} that the left and right duals of invertible objects are isomorphic (so invertible objects have both left and right duals). It then follows from \cite[Proposition 2.10.8]{EGNO} that for any invertible object $U \in \mcC$, the functors $(U \otimes -)$ and $(- \otimes U)$ admit left and right adjoints given by tensoring with the corresponding duals. Any functor between abelian categories with left and right adjoints is exact so the claim follows.

\section{Preservation of Loewy Diagrams}\label{Sec:Loewy}
We consider now exact functors $\mcF:\mcC \to \mcD$ where $\mcC$ is assumed to be locally finite i.e. abelian and $\mathds{k}$-linear such that spaces of morphisms are finite dimensional and every object has finite length. To determine when such a functor preserves the Loewy diagram associated to the socle or radical filtration of $X \in \mcC$, we must first determine when the functor preserves socle or radical filtrations in some appropriate sense. To make this precise, we introduce the following definition.
\begin{definition}\label{Def:pres}
We say that a functor $\mathcal{F}:\mcC \to \mcD$ preserves the socle filtration of $X \in \mcC$ if the Loewy lengths of $X$ and $\mcF(X)$ coincide and
\[ \mcF(\soc_k(X)) \cong \soc_k(\mcF(X))\]
for $k=1,...,n$ with $n$ the Loewy length of $X$.  Similarly, we say that $\mcF:\mcC \to \mcD$ preserves radical filtrations if for any $X \in \mcC$, the Loewy lengths of $X$ and $\mcF(X)$ coincide and
\[ \mcF(\rad_k(X)) \cong \rad_k(\mcF(X))\]
for $k=1,...,n$. We say $\mcF:\mcC \to \mcD$ preserves socle (resp. radical) filtrations if it preserves the socle (resp. radical) filtration of every object in $\mcC$.
\end{definition}
It is clear that if $\mcF:\mcC \to \mcD$ preserves the socle filtration of $X \in \mcC$ then for each $k=1,...,n$
\[ \soc_k(\mcF(X))/\soc_{k-1}(\mcF(X)) \quad  \mathrm{and} \quad \mcF(\soc_k(X)/\soc_{k-1}(X)) \cong \mcF(\soc_k(X))/\mcF(\soc_{k-1}(X))\]
are semisimple with the same simple factors so they are isomorphic and the analogous statement holds for radical filtrations as well. Hence, the functor preserves the layers of the associated Loewy diagram, but perhaps not the arrows. 
\begin{definition}\label{Def:strong}
Denote by $0=M_0 \subset M_1 \subset \hdots \subset M_n=X$ a fixed socle or radical filtration of $X \in \mcC$ and let $X_k$ be a simple factor of $M_k/M_{k-1}$. Given a functor $\mcF:\mcC \to \mcD$ which preserves the socle or radical filtration, we say it preserves the associated Loewy diagram if whenever there exists an arrow (recall Definition \ref{Def:Loewy}) from $X_k$ in layer $k$ to $X_{k-1}$ in layer $k-1$ in the Loewy diagram of $X$, there exists an arrow from $\mcF(X_k)$ in layer $k$ to $\mcF(X_{k-1})$ in layer $k-1$ of the Loewy diagram of $\mcF(X)$.
\end{definition}
Hence, a functor preserves Loewy diagrams precisely when it preserves both the layers and the arrows of the diagrams. We can characterize functors which preserve socle and radical filtrations by how they interact with the socle and radical operations.
\begin{theorem}\label{Thm:Soc}
Let $\mcC$ be locally finite and $\mcF:\mcC \to \mcD$ an exact functor. Then,
\begin{itemize}
\item $\mcF:\mcC \to \mcD$ preserves socle filtrations iff $\mcF(\mathrm{Socle}(X)) \cong \mathrm{Socle}(\mcF(X))$ for all $X \in \mcC$
\item $\mcF:\mcC \to \mcD$ preserves radical filtrations iff $\mcF(\mathrm{Rad}(X)) \cong \mathrm{Rad}(\mcF(X))$ for all $X \in \mcC$
\end{itemize}
\end{theorem}
\begin{proof}
We first prove the statement for radical filtrations. It is clear that if $\mcF:\mcC \to \mcD$ preserves radical filtrations, then $\mcF(\mathrm{Rad}(X)) \cong \mathrm{Rad}(\mcF(X))$ for all $X \in \mcC$. Suppose then that $\mcF(\mathrm{Rad}(X)) \cong \mathrm{Rad}(\mcF(X))$ for all $X \in \mcC$ and let
\[ 0=M_t \subset M_{t-1} \subset \hdots M_1 \subset M_0=X, \quad \mathrm{and} \quad 0 = N_s \subset N_{s-1}\subset \hdots \subset N_1 \subset N_0=\mcF(X) \]
denote the radical filtrations of $X$ and $\mcF(X)$ respectively. We need to show that $t=s$ (recall that $t$ is finite since $\mcC$ is locally finite) and 
\[ \mcF(M_k) \cong N_k\]
for $k=0,...,t$. Notice that $N_1=\Rad(\mcF(X)) \cong \mcF(\Rad(X)) = \mcF(M_1)$ and if $N_{k-1} \cong \mcF(M_{k-1})$, then we have
\[ N_k=\Rad(N_{k-1}) \cong \Rad(\mcF(M_{k-1})) \cong \mcF(\Rad(M_{k-1})) = \mcF(M_k).\]
Hence, $N_k \cong \mcF(M_k)$ for $k=1,...,t$ by induction and $s=t$ since $N_t \cong \mcF(M_t) =0$.

We now consider the statement for socle filtrations. It is clear by definition that any functor preserving socle filtrations satisfies $\mcF(\mathrm{Socle}(X)) \cong \mathrm{Socle}(\mcF(X))$ for all $X \in \mcC$. Let 
\[ 0 = M_0 \subset M_1 \subset \cdots M_t=X \quad \mathrm{and} \quad 0 = N_0 \subset N_1 \subset \cdots N_s=\mcF(X)\]
denote the socle filtrations of $X$ and $\mcF(X)$ respectively, that is, $M_k = \soc_k(X)$ and $N_k =\soc_k(\mcF(X))$. We need to show that $s=t$ (note again that $t$ is finite since $\mcC$ is locally finite) and $\mcF(M_k) \cong N_k$ for $k=0,...,t$. We prove this claim by induction with the $k=1$ case being the assumption $\mcF(\mathrm{Socle}(X)) \cong \mathrm{Socle}(\mcF(X))$. First observe that we have the following equivalence:
\begin{align}
\nonumber \mathrm{Soc}(\mcF(X)/\mcF(M_{k-1})) & \cong \Soc(\mcF(X/M_{k-1}))\\
\label{eq:Sociso}& \cong \mcF(\Soc(X/M_{k-1}))\\
\nonumber& \cong \mcF(M_k/M_{k-1})\\
\nonumber& \cong \mcF(M_k)/\mcF(M_{k-1}).
\end{align}
Since $N_{k-1} \cong \mcF(M_{k-1})$, $N_{k-1}$ is a subobject of $\mcF(M_k)$ and by Remark \ref{Rem} we have
\[ \mcF(X)/\mcF(M_{k-1}) \cong \mcF(X)/N_{k-1} \qquad \mathrm{and} \qquad \mcF(M_k)/N_{k-1} \cong \mcF(M_k)/\mcF(M_{k-1}).\]
It therefore follows from Equation \eqref{eq:Sociso} that we have $\Soc(\mcF(X)/N_{k-1}) \cong \mcF(M_k)/N_{k-1}$ so it follows from Equation \eqref{eq:Soc} that $N_{k} \cong \mcF(M_{k})$ if $N_{k-1} \cong \mcF(M_{k-1})$ so the induction step holds. Since $N_t \cong \mcF(M_t) \cong \mcF(X)$ we have $t=s$.
\end{proof}
We must now determine when an exact functor preserves arrows in the Loewy diagram of $X \in \mcC$. To this end we introduce the following definition. 
\begin{definition}
We call an exact functor $\mcF:\mcC \to \mcD$ strong exact if any short exact sequence $0 \to U \to V \to W \to 0$ in $\mcC$ splits iff the short exact sequence $0 \to \mcF(U) \to \mcF(V) \to \mcF(W) \to 0$ in $\mcD$ induced by $\mcF$ splits.
\end{definition}

\begin{corollary}\label{Cor:Loewy}
Let $\mcC$ be locally finite and $\mcF: \mcC \to \mcD$ an exact functor which preserves socle or radical filtrations. Then $\mcF:\mcC \to \mcD$ preserves the associated Loewy diagrams if it is strong exact.
\end{corollary}
\begin{proof}
Suppose that there exists an arrow from $Y$ in layer $k$ to $Y'$ in layer $k-1$ in the Loewy diagram of $X$. Recall from Definition \ref{Def:Loewy} that this implies that there exists a subquotient $M$ of $M_{k}/M_{k-2}$ such that a non-split short exact sequence 
\[ 0 \to Y' \to M \to Y \to 0\]
exists. By exactness, $\mcF(M)$ is a subquotient of 
\[ \mcF(\soc_k(M))/\mcF(\soc_{k-2}(M)) \cong \soc_k(\mcF(X))/\soc_{k-2}(\mcF(X)) \]
and since $\mcF$ is strong exact we have a non-split short exact sequence 
\[ 0 \to \mcF(Y') \to \mcF(M) \to \mcF(Y) \to 0\]
and $\mcF(Y),\mcF(Y')$ are simple since $\mcF:\mcC \to \mcD$ preserves socle (or radical) filtrations. Hence, there exists an arrow from $\mcF(Y)$ in layer $k$ to $\mcF(Y')$ in layer $k-1$ of the Loewy diagram of $\mcF(X)$. 
\end{proof}

\subsection{Induction Functors}\label{Sec:Ind}

The main result of this section will be sufficient conditions under which the induction functor (recall Equation \eqref{Eq:Ind}) $\mcF_A:\mcC \to \Rep A$ for some commutative algebra object $A \in \mcC$ which is a direct sum over a closed collection of invertible objects preserves Loewy diagrams. That is, we assume
\[A:=\bigoplus\limits_{i \in \mcI} U_i \]
where $\{U_i\}_{i \in \mcI}$ is a closed collection of invertible objects in $\mcC$ as in Subsection \ref{subsec:inv}. We will refer to such commutative algebra objects as commutative simple current algebra objects and we will say $A$ has no fixed points (or is fixed point free) when all $U_i, i\in \mcI$ have no fixed points. We will begin by establishing two necessary lemmas for our main result. First, we determine when the induction functor associated to a commutative simple current algebra object preserves simplicity of objects through a generalization of \cite[Proposition 3.4]{CR} and \cite[Propositions 4.4 \& 4.5]{CKM}:
\begin{lemma}\label{Lem:irred}
Let $\mcC$ be a braided category and $A= \bigoplus_{i \in \mcI} U_i\in \mcC$ a commutative simple current algebra object which is simple as a left $A$-module over itself and has no fixed points. Then we have the following:
\begin{itemize}
\item $\mcF_A:\mcC \to \Rep A$ preserves simplicity.
\item If every object in $\mcC$ contains a simple subobject then $N \in \Rep A$ is simple only if $N \cong \mcF_A(M)$ for some simple $M \in \mcC$.
\item Suppose $A$ contains $\mathds{1}$ as a summand and let $M,N \in \mcC$ be simple, then $\mcF_A(M) \cong \mcF_A(N)$ iff $M \cong U_i \otimes N$ for some $i \in \mcI$.
\end{itemize}
\end{lemma}
\begin{proof}
Let $M \in \mcC$ be simple and suppose  $\mcF_A(M) \in \Rep A$ admits a proper subobject $X \stackrel{\iota}{\hookrightarrow} \mcF_A(M)$. This induces a $\mcC$-morphism via the forgetful functor:
\[ \mcG(X) \stackrel{\mcG(\iota)}{\hookrightarrow} \mcG(\mcF_A(M)) \cong \bigoplus\limits_{i \in \mcI} U_i \otimes M \in \mcC.\]
It follows from Lemma \ref{Lem:inv} that $U_i \otimes M$ is simple for all $i \in \mcI$ so we must have
\[ \mcG(X) \cong \bigoplus\limits_{j \in  \mcJ} U_j \otimes M\]
for some non-empty subset $\mcJ \subset \mcI$. Since $X$ is a subobject of $\mcF_A(M)$ it must be closed under the action of $\mu_{\mcF_A(M)}=\mu \otimes \mathrm{Id}_M$. Since we assume $\{U_i\}_{i \in \mcI}$ is closed and $A$ has no fixed points, $U_i \otimes M \not \cong U_j \otimes M$ for all $i,j \in \mcI$ by Lemma \ref{Lem:inv} so it follows that for all $i \in \mcI,j \in \mcJ$, the image of the restriction $\mu|_{U_i \otimes U_j}$ lies in $A_{\mcJ}:=\bigoplus_{j \in \mcJ} U_j \subset A$. Therefore, $A_{\mcJ}$ is a proper subobject of $(A,\mu) \in \Rep A$, a contradiction.

To prove the second claim, let $N \in \Rep A$ be simple. By assumption, $\mathcal{G}(N) \in \mcC$ has a simple subobject $M$ so there is a non-zero morphism
\[ f:M \to \mcG(N).\]
By Frobenius reciprocity \eqref{Eq:Frob}, there is a non-zero morphism $g:\mcF(M) \to N$ which is an isomorphism since both $N$ and $\mcF(M)$ are simple. It only remins to prove the third claim. First, suppose that $\mcF_A(M)\cong \mcF_A(N)$, then
\[ \bigoplus_{i \in \mcI }U_i \otimes M \cong \bigoplus_{i \in \mcI} U_i \otimes N\]
as objects in $\mcC$. Since $M,N$ are simple, $U_j \otimes N$ is also simple for any $j \in \mcI$ by Lemma \ref{Lem:inv} and we can restrict this isomorphism to $M \subset \mcF_A(M)$ to see that $ M \cong U_j \otimes N$ for some $j \in \mcI$. Conversely, if $M \cong U_j \otimes N$ for some $j \in \mcI$, then by Frobenius reciprocity \eqref{Eq:Frob} we have
\begin{align*}
\dim \Hom_{\Rep A}(\mcF_A(M),\mcF_A(N))&= \dim \Hom_{\mcC}(M,\mcG(\mcF_A(N))) \\
&=\dim \Hom_{\mcC}\left(M,\bigoplus\limits_{i \in \mcI}U_i \otimes N \right) >0
\end{align*}
so $\mcF_A(M) \cong \mcF_A(N)$ since they are simple.
\end{proof}
In addition to the above to lemma, we need to determine when the induction functor is strong exact, which is given in the following lemma.
\begin{lemma}\label{Lem:strongexact}
Let $A\in \mcC$ be a commutative algebra object which contains the unit object $\mathds{1}$ of $\mcC$ as a summand. Then the induction functor $\mcF_{A}:\mcC \to \Rep A$ is strong exact.
\end{lemma}
\begin{proof}
Induction functors are always exact \cite[Theorem 1.6]{KO} so we only need to show that a short exact sequence
\begin{equation}\label{seq1}
0 \to X \overset{f}{\to} Y \overset{g}{\to} Z \to 0
\end{equation}
of objects in $\mcC$ splits if the induced short exact sequence
\begin{equation}\label{seq2}
0 \to \mcF_A(X) \overset{\mcF_A(f)}{\to} \mcF_A(Y) \overset{\mcF_A(g)}{\to} \mcF_A(Z) \to 0 
\end{equation}
splits. Suppose that sequence \eqref{seq2} splits, so there exists a map $h:\mcF_A(Y) \to \mcF_A(X)$ such that $h \circ \mcF_A(f)=\mathrm{Id}_{\mcF_A(X)}$. The induction functor acts by $X \to (A \otimes X, \mathrm{Id}_{A} \otimes f)$, so viewing sequence \eqref{seq2} as a sequence in $\mcC$, we have 
\[ h \circ (\mathrm{Id}_{A} \otimes f)=\mathrm{Id}_{A \otimes X}. \]
Since $A$ contains the unit object as a summand, we have $A= \mathds{1} \oplus A'$ for some object $A' \in \mcC$. Let
\[ \iota_{\mathds{1}}:\mathds{1} \to A, \quad \iota_{A'}:A' \to A, \qquad p_{\mathds{1}}:A \to \mathds{1}, \quad p_{A'}:A \to A'\]
be the associated inclusion and projection maps. Acting by $p_{\mathds{1}} \otimes \mathrm{Id}_X$ on the left, $\iota_{\mathds{1}} \otimes \mathrm{Id}_X$ on the right, and substituting $\mathrm{Id}_A=\iota_{\mathds{1}} \circ p_{\mathds{1}} + \iota_{A'} \circ p_{A'}$ we see that
\begin{align*}
\mathrm{Id}_{\mathds{1} \otimes X}&=p_{\mathds{1}} \otimes \mathrm{Id}_X \circ h \circ (\iota_{\mathds{1}} \circ p_{\mathds{1}}+\iota_{A'} \circ p_{A'}) \otimes f \circ \iota_{\mathds{1}} \otimes \mathrm{Id}_X\\
&=p_{\mathds{1}} \otimes \Id_X \circ h \circ \iota_{\mathds{1}} \otimes f\\
&= \tilde{h} \circ \Id_{\mathds{1}} \otimes f
\end{align*}
where $\tilde{h}=p_{\mathds{1}} \otimes \Id_X \circ h \circ \iota_{\mathds{1}} \otimes \Id_Y$. Composing this equality with the left unit isomorphism $\ell_X$ on the left, $\ell_X^{-1}$ on the right, and applying naturality gives
\begin{align*}
\mathrm{Id}_X&=\ell_X \circ \tilde{h} \circ (\mathrm{Id}_{\mathds{1}} \otimes f) \circ \ell_X^{-1}\\
&= (\ell_X \circ \tilde{h} \circ \ell_Y^{-1}) \circ f
\end{align*}
Therefore, the map $\ell_X \circ \tilde{h} \circ \ell_Y^{-1}:Y \to X$ provides a splitting of sequence \eqref{seq1}.
\end{proof}
 We can now apply these lemmata to prove the main theorem of this section. 
\begin{theorem}\label{Thm:Ind}
Let $\mcC$ be a locally finite braided category and $A\in \mcC$ a commutative simple current algebra object which is simple as a left $A$-module and has no fixed points. Then the induction functor $\mcF_A:\mcC \to \Rep A$ preserves Loewy diagrams associated to socle and radical filtrations.
\end{theorem}
\begin{proof}
$\mcF_{A}$ is exact and preserves simplicity by Lemma \ref{Lem:irred}, so for all $X \in \mcC$, there are embeddings
\[ \mcF_A(\mathrm{Socle}(X)) \hookrightarrow \mathrm{Socle}(\mcF_A(X)) \quad  \mathrm{and} \quad \mcF_A(\Rad(X)) \hookrightarrow \Rad(\mcF_A(X)). \]
To show this embedding is an isomorphism, it is enough to show that each of the above (semi-simple) modules contain the same number of simple factors. We prove this for the first embedding as the second is analogous. It is enough to prove the following equality:
\[ \sum\limits_{\mcF_A(S) \in I(\Rep A)} \mathrm{dim}\mathrm{Hom}_{\Rep_A}(\mcF_A(S),\mcF_A(X)) = \sum\limits_{S \in I(C)} \mathrm{dim}\mathrm{Hom}_{\mcC}(S,X) \]
where $I(\mcC)$ and $I(\Rep A)$ denote the collection of all isomorphism classes of simple objects in $\mcC$ and $\Rep A$ respectively. Let $J \subset I(\mcC)$ be any collection of isomorphism classes of simple objects in $\mcC$ such that for every simple object $\mcF_A(S) \in I(\Rep A)$ (recall from Lemma \ref{Lem:irred} that all simple objects in $\Rep A$ take this form), $J$ contains exactly one simple $S' \in \mcC$ such that $\mcF_A(S')\cong \mcF_A(S)$. In particular, we have $|J|=|I(\Rep A)|$ and $J_{\mcI}:=\{ U_{-i} \otimes S \, | \, i \in \mcI, \; S \in J\}=I(\mcC)$ since two simple objects $S$ and $S'$ induce to isomorphic objects iff $S'=U_i \otimes S$ (i.e. $U_{-i} \otimes S' \cong S$) for some $i \in \mcI$ by Lemma \ref{Lem:irred}. It now follows from Frobenius reciprocity \eqref{Eq:Frob} that
\begin{align*}
\sum\limits_{\mcF_A(S) \in I(\Rep A)} \dim \Hom_{\Rep A}(\mcF_A(S),\mcF_A(X))& =\sum\limits_{S \in J} \dim \Hom_{\mcC}(S, \mcG(\mcF_A(X)))\\
&=\sum\limits_{S \in J} \dim \Hom_{\mcC}(S,\bigoplus\limits_{i \in \mcI} U_i \otimes X)\\
&=\sum\limits_{S \in J, i \in \mcI} \dim \Hom_{\mcC}(S,U_i \otimes  X)\\
&= \sum\limits_{S \in J, i \in \mcI} \dim \Hom_{\mcC}(U_{-i} \otimes S,X)\\
&=\sum\limits_{S \in J_{\mcI}} \dim \Hom_{\mcC}(S,X)\\
&=\sum\limits_{S \in I(\mcC)} \dim \Hom_{\mcC}(S,X).
\end{align*}
where we have used the fact that each $U_i$ is rigid with dual $U_{-i}$. Therefore, $\mcF_A:\mcC \to \Rep A$ preserves socle filtrations by Theorem \ref{Thm:Soc}. By assumption, $\{U_i\}_{i \in \mcI}$ is closed, so $\mathds{1} \cong U_i^* \otimes U_i$ appears as a summand of $A$ and $\mcF_A$ preserves Loewy diagrams by Lemma \ref{Lem:strongexact}.
\end{proof}

\end{document}